\documentclass[12pt]{amsart}

\usepackage{amsmath,amssymb,amsxtra,anysize,times,tikz, float,combelow}
\marginsize{1in}{1in}{1in}{1in}

\newtheorem{lemma}{Lemma}[section]
\newtheorem{theorem}[lemma]{Theorem}%[section]
\newtheorem{conjecture}[lemma]{Conjecture}%[section]
\newtheorem{corollary}[lemma]{Corollary}%[section]
\newtheorem{proposition}[lemma]{Proposition}%[section]
\theoremstyle{definition}   %\MV{seeing how it looks with this style}
\newtheorem{example}[lemma]{Example}%[section]
\newtheorem{remark}[lemma]{Remark}%[section]
\newtheorem{definition}[lemma]{Definition}%[section] 
 
\DeclareMathOperator{\An}{\mathcal{A}} %generalized Andersen map
\DeclareMathOperator{\Core}{Core}
\DeclareMathOperator{\Quot}{Quot}
\DeclareMathOperator{\dinv}{dinv}

\DeclareMathOperator{\hgt}{ht}
\DeclareMathOperator{\PF}{\mathcal{PF}}
\DeclareMathOperator{\area}{area}
\DeclareMathOperator{\ch}{ch}
\DeclareMathOperator{\SSRT}{SSRT}
\DeclareMathOperator{\SRT}{SRT}

\DeclareMathOperator{\spin}{spin}
\DeclareMathOperator{\ides}{ides}
\DeclareMathOperator{\ias}{ias}

\DeclareMathOperator{\des}{des}
\DeclareMathOperator{\SSYT}{SSYT}
\DeclareMathOperator{\SYT}{SYT}
\DeclareMathOperator{\st}{st}

\newcommand{\Sn}{S_n}
\newcommand{\aSmn}{\widetilde{S}_n^m} %m-stable part of AFFINE symm group
\newcommand{\BC}{\mathbb{C}}
\newcommand{\BZ}{\mathbb{Z}}
\newcommand{\aSn}{\widetilde{S}_n}
\newcommand{\CF}{\mathcal{F}}
\newcommand{\om}{\omega}

\title{Rational parking functions and LLT polynomials}
\author{Eugene Gorsky}
\address{Department of Mathematics, Columbia University, 2990 Broadway, New York NY 10027}
\address{Department of Mathematics, UC Davis, One Shields Avenue, Davis CA 95616}
\address{International Laboratory of Representation Theory and Mathematical Physics, NRU-HSE, Vavilova 7, Moscow, Russia}
\email{egorsky@math.columbia.edu}
\author{Mikhail Mazin}
\address{Mathematics Department, Kansas State University, 138 Cardwell Hall,
Manhattan, KS 66506}
\email{mmazin@math.ksu.edu}

\begin{document}
\begin{abstract}
We prove that the combinatorial side of the ``Rational Shuffle Conjecture'' provides a Schur-positive symmetric polynomial.
Furthermore, we prove that the contribution of a given rational Dyck path can be computed as a certain skew LLT polynomial, thus generalizing the result of Haglund, Haiman, Loehr, Remmel and Ulyanov. The corresponding skew diagram is described explicitly in terms of a certain $(m,n)$--core.
\end{abstract}

\maketitle

\section{Introduction}

The space of diagonal coinvariants $DH_n$ is defined as the quotient of the polynomial ring $\BC[x_1,\ldots,x_n,y_1,\ldots,y_n]$ with respect to the ideal generated by the positive degree invariants of the diagonal action of the symmetric group $S_n$. This space is naturally
bigraded by degrees in $x$- and in $y$-variables, and carries a degree preserving $S_n$-action. In the series of papers \cite{haiman1,haiman2,haiman3,haiman4} Haiman proved that the dimension of $DH_n$ equals $(n+1)^{n-1}$
and its bigraded Frobenius character  equals $\nabla e_n$, where $\nabla$ is a certain operator on symmetric functions which diagonalizes in the basis of modified Macdonald polynomials.

Finding an explicit basis in $DH_n$ remains an important open problem in algebraic combinatorics. Recall that a {\em parking function} on $n$ cars is a map 
$f:\{1,\ldots,n\}\to \BZ_{\ge 0}$ such that $\sharp f^{-1}([0,i-1])\ge i$ for all $i$. Let $\PF_n$ denote the set of parking functions. It is well known that its cardinality equals
$\sharp \PF_n=(n+1)^{n-1}=\dim DH_n.$ 
%so one may seek a combinatorial formula for the Frobenius character of $DH_n$ in the form of the sum over parking functions. 
Moreover, $S_n$ acts on $\PF_n$ by permuting the values, and this action preserves a natural statistic on parking functions: $\area(f):=\frac{n(n-1)}{2}-\sum f(i).$ It follows from the work of Garsia and Haiman (\cite{GH96}) that if one forgets one of the gradings on $DH_n,$ then $DH_n$ is isomorphic to the space $\BC\PF_n,$ tensored by the sign representation, as a graded $S_n$-module.
In \cite{HHLRU}, Haiman, Haglund, Loehr, Remmel and Ulyanov proposed a conjectural formula for the bigraded Frobenius characteristic of $DH_n$, which became known as {\em Shuffle Conjecture:}

\begin{conjecture}(\cite{HHLRU})
The following equation holds:
\begin{equation}
\label{shuffle}
\ch DH_n=\nabla e_n=\sum_{f\in \PF_n}q^{\area(f)}t^{\dinv(f)}\cdot Q_{\ias(f)}(z),
\end{equation}
where $\dinv$ is a certain statistic on parking functions, $\ias(f)$ is the set of ascents of the inverse of the diagonal word of $f,$ and $Q_{\ias(f)}(z)$ is the Gessel fundamental quasisymmetric function in variables $\{z_1,z_2,\ldots \}$ (see \cite{G}).
\end{conjecture}

%More recently, the following ``rational'' version of parking functions has been defined (\cite{arms,GMV}):

%\begin{definition}
%A function $f:\{1,\ldots,m\}\to \{1,\ldots,n\}$ is called an $m/n$--parking function if for all $i$ the following inequality holds:
%$\sharp f^{-1}(\{1,\ldots,i\})\ge \frac{m}{n}i$. The set of rational parking functions will be denoted by $\PF_{m/n}$. 

%Given $f\in \PF_{m/n}$ one can consider a partition obtained by sorting of $f(i)$. The corresponding Young diagram defines a lattice path in the $m\times n$ rectangle below the diagonal, which will be called the $m/n$--Dyck path underlying $f$.
%\end{definition}

\begin{remark}
In the original formulas in \cite{HHLRU} one has the set of descents of the diagonal word $\ides(f)$ instead of the ascents $\ias(f).$ This is due to a difference in notations for parking functions (see Figure \ref{figure: parking function} for an example).
\end{remark}

Tensoring by the sign representation corresponds to the involution $\Omega$ on the space of symmetric functions defined by $\Omega(s_\lambda)=s_{\lambda'},$ where $\lambda'$ is obtained from $\lambda$ by transposition. Involution $\Omega$ can also be extended to the space of quasisymmetric functions by setting $\Omega(Q_S)=Q_{\overline{S}},$ where $\overline{S}=\{1,\ldots,n-1\}\setminus S.$
In particular, $\Omega(Q_{\ias(f)})=Q_{\ides(f)}.$

In \cite{gn}, a rational analogue of the Shuffle Conjecture has been proposed.

\begin{conjecture}
The following equation holds:
\begin{equation}
\label{rat shuffle}
P_{m,n}\cdot 1=\sum_{f\in \PF_{m/n}}q^{\area(f)}t^{\dinv(f)}\cdot Q_{\ides(f)}(z),
\end{equation}
where $P_{m,n}$ is a certain degree $n$ operator acting on symmetric functions, $\PF_{m/n}$ is a rational analog of the parking functions (see definition below), and $\area$ and $\dinv$ generalize the corresponding statistics to rational parking functions.
\end{conjecture}

It is also conjectured (see \cite{gn} for a detailed exposition) that \eqref{rat shuffle} is equal to the bigraded Frobenius character of the unique finite-dimensional
representation $L_{m/n}$ of the rational Cherednik algebra. In particular, both sides of this equation are expected to have nonnegative coefficients in the Schur expansion.

In this article we prove that the right hand side of \eqref{rat shuffle} is indeed Schur--positive. More precisely, let $\PF_{m/n}(D)$ denote the set of $m/n$--parking functions with
the underlying Dyck path $D$. Note that the $\area$ statistic is constant on $\PF_{m/n}(D)$.

\begin{theorem}
\label{thm:main}
For all $m/n$--Dyck paths $D$ the polynomial
$$
\CF(D;t):=\sum_{f\in \PF_{m/n}(D)}t^{\dinv(f)}\cdot Q_{\ides(f)}(z)
$$
is a symmetric Schur positive polynomial.
\end{theorem}

\begin{corollary}
The combinatorial side of the ``rational Shuffle Conjecture'' equals 
$$
\CF_{m/n}(q,t)=\sum_{D}q^{\area(D)}\CF(D;t)
$$ 
and hence is symmetric and Schur positive.
\end{corollary}

Theorem \ref{thm:main} was first proved in \cite[Corollary 4.16]{Hikita} using the geometry of affine Springer fibers.
Our proof of Theorem \ref{thm:main} follows the ideas of \cite{HHLRU}: we prove that the coefficients of $\CF(D;t)$
in the Schur expansion (up to a monomial shift)  can be identified with certain parabolic affine Kazhdan--Lusztig polynomials
labeled by a certain partition $\mu$ and its $m$-core $\lambda$. We use a bijection of J. Anderson to give a simple construction 
of $\lambda$ and $\mu$ which seems to clarify the subtle combinatorial considerations of \cite{HHLRU}.

\section*{Acknowledgments}

We would like to thank Fran\c{c}ois Bergeron, Andrei Negu\cb t and Monica Vazirani for the useful discussions.  
The research of E.G. was partially supported by the grants DMS-1403560 and RFBR-13-01-00755. % and by the Government of the Russian Federation 
%within the framework of the implementation of the 5-100 Program Roadmap of the National Research University Higher School of Economics.

\section{Definitions and notations}

All Young diagrams are in "French notation".  A set $M\subset \BZ$ is called $n$--invariant if for each $x\in M$ one has $x+n\in M$. 
%For a subset $M\subset \BZ$ we define its conductor $\cond M$ as an interval of the form $[x,+\infty)\subset M$ with minimal possible $x$ (if there is no such $x$, we define $\cond M=\emptyset$). 
%Following \cite{HHLRU}, we denote the set of semistandard Young tableaux of shape $\lambda$ by $\SSYT(\lambda)$ and the set of semistandard $m$-ribbon  tableaux of shape $\lambda$ by $\SSRT_m(\lambda)$.

\subsection{Rational parking functions}

For a function $f:\{1,\ldots,n\}\to \BZ_{\ge 0},$ let $D_f$ denote the Young diagram with the row lengths equal to the values of $f$ put in decreasing order.

\begin{definition}
A function $f$ is called an {\em $m/n$--parking function} if the diagram $D_f$ fits under the diagonal in an $n\times m$ rectangle. The set of $m/n$--parking function is denoted $\PF_{m/n}.$ 
\end{definition}

Equivalently, $f\in\PF_{m/n}$ if and only if for all $1\le i\le m$ the following inequality holds:

$$
\sharp f^{-1}(\{0,\ldots,i-1\})\ge \frac{n}{m}i.
$$

\begin{definition}
The Young diagram $D_f$ will be called the $m/n$--Dyck path underlying $f$. For any $m/n$--Dyck path $D,$ let $\PF_{m/n}(D)\subset\PF_{m/n}$ denote the set of $m/n$-parking functions with the underlying Dyck path $D.$
\end{definition}

\begin{remark}
Note that a Young diagram fits under the diagonal in an $n\times (n+1)$ rectangle if and only if it fits under the diagonal in an $n\times n$ square. Thus, the case $m=n+1$ corresponds to classical parking functions $\PF_{(n+1)/n}=\PF_n.$
\end{remark}

Another way to think about parking functions is to identify the set $\PF_{m/n}(D)$ with the set of standard Young tableaux of the skew shape $D+(1)^n\setminus D$ denoted $\SYT(D+(1)^n\setminus D).$ To recover the function from such a tableau one sets $f(i)$ equal to the length of the row containing the label $i.$ The monotonicity condition for columns insures that we account for each parking function exactly once. It will be convenient for us to assume that the labels decrease in columns from bottom to top %. The right way to switch to the increasing notation is to switch the labels according to the rule $i\mapsto n+1-i$ 
(see Figure \ref{figure: parking function} for an example).

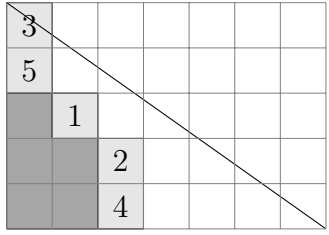
\begin{figure}
\begin{tikzpicture}[scale=0.6]

\filldraw  [fill=gray!70!white](0,0)--(0,3)--(1,3)--(1,2)--(2,2)--(2,0)--(0,0); 
\filldraw  [fill=gray!20!white](0,3)--(0,5)--(1,5)--(1,3)--(0,3); 
\filldraw  [fill=gray!20!white](1,3)--(2,3)--(2,2)--(1,2)--(1,3); 
\filldraw  [fill=gray!20!white](2,0)--(2,2)--(3,2)--(3,0)--(2,0); 
%\draw  [thick] (0,3)--(0,5)--(1,5)--(1,3)--(2,3)--(2,2)--(3,2)--(3,0)--(2,0);
\draw  (0,5)--(7,0); 
\draw [step=1, color=gray, very thin] (0,0) grid (7,5);

\draw (2.5,0.5) node {$4$};
\draw (2.5,1.5) node {$2$};
\draw (1.5,2.5) node {$1$};
\draw (0.5,3.5) node {$5$};
\draw (0.5,4.5) node {$3$};

\end{tikzpicture}
\caption{Consider the function $f:\{1,2,3,4,5\}\to \BZ_{\ge 0}$ given by $f(1)=1,\ f(2)=2,\ f(3)=0,\ f(4)=2,\ f(5)=0.$ The underlying diagram $D_f=\{2,2,1\}$ fits under the diagonal in a $5\times 7$ rectangle, so $f\in\PF_{7/5}$. The figure shows the corresponding standard Young tableau of shape $D_f+(1)^n\setminus D_f.$}\label{figure: parking function}
\end{figure} 

\subsection{Affine permutations}\label{Section: Affine permutations}

We will need a bijection between rational parking functions and a subset in the affine symmetric group $\aSn,$ constructed in \cite{GMV}.

\begin{definition}
A bijection $\omega:\BZ\to \BZ$ is called an affine $S_n$--permutation, if $\omega(x+n)=\omega(x)+n$ for all $x,$ and $\sum_{i=1}^{n}\omega(i)=\frac{n(n+1)}{2}.$ The set of affine $S_n$--permutations form a group with respect to composition. The group is called the {\em affine symmetric group} and denoted $\aSn.$
\end{definition}

\begin{definition}
An affine permutation $\omega\in \aSn$ is called 
{\em $m$-stable}
if for all $x$ the inequality $\omega(x+m)>\omega(x)$ holds,
i.e. $\om$ has no inversions of height $m$.
The set of all $m$-stable affine permutations is denoted $\aSmn$.
\end{definition}

Let us briefly recall the construction of the bijection $\An:\aSmn\to \PF_{m/n}$ (see \cite{GMV} for more details). Take a permutation $\om\in\aSmn.$ Consider the set $\Delta_{\omega}:=\{i\in\BZ: \omega(i)>0\}\subset\BZ$ and let $M_{\om}$ be its minimal element. Note that the set $\Delta_{\omega}$ is invariant under addition of $m$ and $n.$ %Indeed, if $i\in \Delta_{\w}$ then $\w(i+m)>\w(i)>0$ and $\w(i+n)=\w(i)+n>n>0.$ Therefore $i+m\in\Delta_\w$ and $i+n\in\Delta_\w.$
Let us label the boxes in the $n\times m$ rectangle $R_{m,n}$ so that the box $(i,j)$ is labeled by the weight $l(i,j)=mn-m-n+M_{\om}-mi-nj$ (assuming that the bottom-left corner box has coordinates $(0,0)$). We also extend this labeling to the whole $\BZ^2$ when needed.
%Consider the integer lattice $\mathbb{Z}^2.$ We prefer to think about it as of the set of square boxes, rather than the set of integer points. Consider the rectangle $R_{m,n}:=\{(x,y)\in \mathbb{Z}^2 \mid 0\le x<m, 0\le y<n\}.$ Let us label the boxes of the lattice according to the linear function 
%$$
%l(x,y):=(mn-m-n)+M_{\w}-nx-my.
%$$ 
The function $l(i,j)$ is chosen in such a way that a box is labeled by $M_\om$ if and only if its top-right corner touches the line containing the top-left to bottom-right diagonal of the rectangle, so $l(i,j)\ge M_{\om}$ if and only if the box $(i,j)$ is below this line. The  Young diagram $D_{\om}$ is defined by 
$$
D_{\om}:=\{(i,j)\in R_{m,n}\mid  l(i,j) \in \Delta_{\om}\}.
$$
%If $(x,y)\in D_{\w},$ then $\omega(l(x,y))>0$, hence
%$$
%\omega(l(x-1,y))=\omega(l(x,y)+n)>0,
%$$
%and
%$$
%\omega(l(x,y-1))=\omega(l(x,y)+m)>\omega(l(x,y))>0.
%$$
%Therefore, if $x-1\ge 0,$ then $(x-1,y)\in D_{\w},$ and if $y-1\ge 0,$ then $(x,y-1)\in D_{\w}$. We conclude that $D_{\w}\subset R_{m,n}$ is indeed a Young diagram with the SW corner box $(0,0).$ Note also that $D_\w$ fits under the NW-SE diagonal of $R_{m,n}.$ Therefore, $D_w\in\Ymn.$

The diagram $D_{\om}$ will be the underlying $m/n$--Dyck path of the parking function $\An_\om:=\An(\om),$ so that the set of values of $\An_\om$ equals the set of row lengths of $D_\om.$ What remains to do is to assign the arguments to the values. This is done by labeling the $i^{\mathrm{th}}$ row of the diagram by $\om(a_i),$ where $a_i$ is the weight of the rightmost box of the $i^{\mathrm{th}}$ row of $D_{\om}$ (if a row has length $0$ we take the weight of the box $(-1,i-1),$ just outside the rectangle in the same row).

Note that the weights $\{a_1,\ldots,a_n\}$ of the rightmost boxes of the diagram $D_\om$ are the smallest elements of the set $\Delta_\om\subset\BZ$ in their corresponding congruence classes modulo $n$ (i.e. $a_i-n\notin \Delta_\om$ for all $i$). They are called the {\em $n$--generators of $\Delta_\om.$} It follows that $$\{\om(a_1),\ldots,\om(a_n)\}=\{1,\ldots, n\}.$$

\begin{lemma}[\cite{GMV}]
The map $\An:\aSmn\to\PF_{m/n}$ is a bijection.
\end{lemma}

Let $f\in\PF_{m/n}$ and $\om\in\aSmn$ be such that $\An_\om=f.$ Let $\Delta_{\omega}:=\{i\in\BZ: \omega(i)>0\}\subset\BZ$ and $M_{\om}=\min \Delta_\om$ as above.

\begin{definition}
 Following \cite{GMV}, we define two statistics on parking functions : 
$$
\area(\om)=\area(f):=\frac{(m-1)(n-1)}{2}-\sum_{a=1}^{n}f(a)=\sharp \left([M_\om,+\infty)\setminus \Delta_\om\right)=1-M_\om,
$$
and 
$$
\dinv(\om):=\dinv(f):=\{ (i,j) \in \BZ^2 \mid 1\le i\le n,\ i < j < i+m,  \om(i) > \om(j) \}.
$$

%$$
%\dinv(f)=\frac{(m-1)(n-1)}{2}-\sum_{a=1}^{n}\sharp\left\{\beta >a: 
%0 < \om_f(a) -  \om_f(\beta)<m \right\}
%$$
\end{definition}

\begin{remark}
The equivalence of the formulas for $\area(f)=\area(\om)$ can be shown as follows. By the first formula we get that $\area(f)$ is the number of boxes that fit under the diagonal in the rectangle $R_{m/n},$ but don't fit in the diagram $D_\om.$ The set of weights of these boxes is exactly $[M_\om,+\infty)\setminus \Delta_\om,$ each occurring once, which proves equivalence of the first two formulas. Furthermore, observe that the set $\Delta_\om$ is always ``balanced'': there are as many non-positive elements in $\Delta_\om$ as there are positive elements of the complement. This proves the third formula. Note that $\area(f)$ is constant on $\PF_{m/n}(D)$ for any $m/n$--Dyck path $D$.

\end{remark}

\begin{remark}
The statistic $\dinv(\om)$ basically counts the inversions of $\om$ of height less than $m.$ Describing $\dinv(f)$ directly in terms of the parking function $f$ is somewhat complicated (see \cite{Hikita}).  
\end{remark}

Note that the diagram $D_{\om}$ depends only on the set $\Delta_\om,$ and, vice versa, for any two permutations $\om_1,\om_2\in\An^{-1}(\PF_{m/n}(D))$ one has $\Delta_{\om_1}=\Delta_{\om_2}.$ We will use the following notations:

\begin{definition}
Let $\Delta_D\subset \BZ$ denote the subset given by $\Delta_D=\Delta_\om$ for any $\om\in\aSmn$ such that $\An(\om)\in\PF_{m/n}(D).$ Let also 
$$
\aSmn(\Delta):=\{\om\in \aSmn\ |\ \Delta_\om=\Delta\}=\An^{-1}(\PF_{m/n}(D_{\om})).
$$
%Equivalently, $\aSmn(\Delta_D)=\An^{-1}(\PF_{m/n}(D_{\om})).$
\end{definition}

Let us also reinterpret the quasi-symmetric function $Q_{\ides(f)}$ in terms of the affine permutation $\om.$ The diagonal word $dw(f)$ of the parking function $f$ is the word obtained by reading the labels of the corresponding standard tableau in the order given by the weights of the boxes, or, equivalently, the distance from the diagonal of the $n\times m$ rectangle to the left-top corner of the box. Since each number from $\{1,\ldots,n\}$ appears in $dw(f)$ exactly once, one get $dw(f)\in \Sn.$ It follows immediately from the construction, that the descents of the inverse of the diagonal word $dw(f)$ are exactly the same as the descents of the word $(\om^{-1}(1)\om^{-1}(2)\ldots\om^{-1}(n)).$ For simplicity, we denote this set of descents $\des(\om^{-1})$ (see Example \ref{Example: quasi-symm}).

\begin{example}\label{Example: quasi-symm}
Continuing the example in Figure \ref{figure: parking function}, one gets the diagonal word $dw(f)=(35124),$ with the inverse $dw(f)^{-1}=(34152).$ Therefore, the descent set is $\ides(f)=\{2,4\}.$
To recover the corresponding affine permutation $\om\in\aSmn,$ one should first recover the weight labeling on the rectangle. According to the formulas for the $\area$ statistic, one gets
$$
1-M_\om=\frac{(5-1)(7-1)}{2}-\sum f(a)=12-5=7,
$$
so $\min(\Delta_\om)=M_\om=1-7=-6.$  The weight labeling $l(i,j)=17-7i-5j$ is shown on Figure \ref{Figure: weight labeling}. 
%Therefore, the weight labeling $l(i,j)=mn-m-n+M_\om-mi-nj=17-7i-5j$ is as on Figure \ref{Figure: weight labeling}. 
\begin{figure}
\begin{tikzpicture}[scale=0.6]

\filldraw  [fill=gray!50!white](0,0)--(0,3)--(1,3)--(1,2)--(2,2)--(2,0)--(0,0); 
%\draw  (0,0)--(0,5)--(7,5)--(7,0)--(0,0);
%\draw  (1,5)--(1,3)--(2,3)--(2,2)--(3,2)--(3,0);
\draw  (0,5)--(7,0); 
\draw [step=1, color=gray, very thin] (0,0) grid (7,5);

\draw (-2.5,0.5) node {$\bf 4$};
\draw (-2.5,1.5) node {$\bf 2$};
\draw (-2.5,2.5) node {$\bf 1$};
\draw (-2.5,3.5) node {$\bf 5$};
\draw (-2.5,4.5) node {$\bf 3$};

\draw (-0.5,0.5) node {$22$};
\draw (-0.5,1.5) node {$15$};
\draw (-0.5,2.5) node {$8$};
\draw (-0.5,3.5) node {$1$};
\draw (-0.5,4.5) node {$-6$};
\draw (0.5,0.5) node {$17$};
\draw (0.5,1.5) node {$10$};
\draw (0.5,2.5) node {$3$};
\draw (0.5,3.5) node {$-4$};
%\draw (0.5,4.5) node {$-11$};
\draw (1.5,0.5) node {$12$};
\draw (1.5,1.5) node {$5$};
\draw (1.5,2.5) node {$-2$};
\draw (2.5,0.5) node {$7$};
\draw (2.5,1.5) node {$0$};
\draw (3.5,0.5) node {$2$};
\draw (3.5,1.5) node {$-5$};
\draw (4.5,0.5) node {$-3$};

\end{tikzpicture}
\caption{The weight labeling corresponding to the parking function $f\in\PF_{7/5}$ given by $(1,2,3,4,5)\mapsto (1,2,0,2,0).$ We put the row labeling from the corresponding standard Young tableau (see Figure \ref{figure: parking function}) on the left to avoid confusion.}\label{Figure: weight labeling}
\end{figure} 
We conclude that the $5$--generators of  $\Delta_\om$ are $(-6,1,3,5,12),$ and $\om$ is defined by $\om(-6)=3,\ \om(1)=5,\ \om(3)=1,\ \om(5)=2,$ and $\om(12)=4.$ One gets:
$$
(\om^{-1}(1),\om^{-1}(2),\om^{-1}(3),\om^{-1}(4),\om^{-1}(5))=(3,5,-6,12,1).
$$
Note that the descents are the same as for the inverse of the diagonal word: $\des(\om^{-1})=\{2,4\}.$ By definition of the Gessel's fundamental quasi-symmetric function (see \cite{G}):
$$
Q_{\ides(f)}=Q_{\des(\om^{-1})}=\sum\limits_{\substack{i_1\le i_2\le\ldots\le i_5,\\ i_k=i_{k+1}\Rightarrow k\notin\des(\om^{-1})}}z_{i_1}z_{i_2}z_{i_3}z_{i_4}z_{i_5}=\sum\limits_{i_1\le i_2< i_3\le i_4< i_5}z_{i_1}z_{i_2}z_{i_3}z_{i_4}z_{i_5}.
$$
From Figure \ref{Figure: weight labeling} it is clear that $\area(\om)=7.$ To compute $\dinv(\om),$ it is convenient to present $\om$ in the following form:
$$
\begin{array}{ccccccccccccc}
x \ \ &1&2&3&4&5&6&7&8&9&10&11&12\ \\
\omega(x) \ &5&-6&1&13&2&10&-1&6&18&7&15&4\ \\
\end{array}.
$$
One gets 
$$
\dinv(\om)=12-\sharp\{(1,2),(1,3),(1,5),(1,7),(3,7),(4,5),(4,6),(4,7),(4,8),(4,10),(5,7)\}
$$
$$
=12-11=1.
$$
\end{example} 

Combining the above, we get the following identity:
\begin{equation}\label{Equation: Fdt via permutations}
\CF(D;t)=\sum_{f\in \PF_{m/n}(D)}t^{\dinv(f)}\cdot Q_{\ides(f)}(z)=\sum\limits_{\om\in\aSmn(\Delta_D)} t^{\dinv(\om)}Q_{\des(\om^{-1})}(z).
\end{equation}
One can also reformulate the rational Shuffle conjecture:

\begin{conjecture}
The following equation holds:
$$
P_{m,n}\cdot 1=\sum_{\om\in \aSmn}q^{\area(\om)}t^{\dinv(\om)}\cdot Q_{\des(\om^{-1})}(z),
$$
\end{conjecture}

\subsection{Ribbon tableaux}

Given a Young diagram $\mu$, let us go along its boundary from the bottom-right to the top-left, and write 0 if we go left and 1 if we go up. We get a sequence of 0's and 1's which stabilizes to 0 at $-\infty$ and to 1 at $+\infty$. Such a sequence is sometimes referred to as ``Maya diagram'' \cite{Johnson,Nagao}, and can be interpreted  as the characteristic function of a subset $M(\mu)$ in $\BZ$. The subset $M(\mu)$ is defined up to a shift. The standard way to choose a representative is as follows. 

\begin{definition}
Given a box $(i,j)\in \BZ^2$ we say that its content equals $j-i.$ The set $M(\mu)$ is defined as the set of contents of all boxes to the left of the 
vertical steps in the boundary of $\mu$.
\end{definition}

%A horizontal (vertical) edge of the boundary of $\mu$ corresponds to the content of the box above (to the left) of the step. 
In particular, in this normalization the empty diagram corresponds to the subset $\BZ_{>0}.$
Let us recall some standard definitions.

\begin{definition}
A set of boxes $\nu\subset \BZ^2$ is called a {\it skew Young diagram} if there exist Young diagrams $\mu\supset\lambda$ such that $\nu=\mu\backslash\lambda.$ A {\it ribbon of length $m$} (or simply an {\it $m$-ribbon}) is a connected skew Young diagram $\nu$ of area $m$ with no $2\times 2$ squares inside. The content $c(\nu)$ of an $m$-ribbon $\nu$ is the maximum of contents of its boxes. 
A skew Young diagram tiled by several $m$-ribbons is called a {\it skew $m$-ribbon diagram}.
\end{definition}

Suppose that for Young diagrams $\lambda\subset \mu$ the skew shape $\nu=\mu\setminus \lambda$ is an $m$--ribbon.
It is easy to see that $M(\mu)$ is obtained from $M(\lambda)$ as follows: an element $x\in M(\lambda)$ is replaced by $x-m$ (so it ``jumps'' by $m$ units to the left) and all other elements stay unchanged. Note also that $x=c(\nu)+1.$ The following statement is clear from the construction.

\begin{proposition}
\label{prop:jump vs hgt}
Suppose that $\lambda$, $\mu,\nu$ and $x$ are as above. Then the height of the ribbon $\nu$ equals:
$$
\hgt(\nu)=1+\sharp\{y\in M(\lambda): x-m<y<x\}.
$$
\end{proposition}

Note that $\{y\in M(\lambda): x-m<y<x\}$ are exactly the elements of $M(\lambda)$ that $x$ ``jumped over'' as it moved to $x-m.$

\begin{definition}
The spin of an $m$-ribbon $\nu=\mu\backslash\lambda$ is defined as $\spin(\nu)=\hgt(\nu)-1.$ The spin of an $m$-ribbon diagram is the sum of spins of the ribbons of the diagram.
\end{definition}
 
%Statistic $\spin$ appears in the work of Leclerc, Lascoux, and Thibon \cite{LLT}. 
 
\begin{example}
Suppose that $\lambda=\emptyset$, and $\mu=(m)$. Then $M(\lambda)=[1,+\infty)$ and $M(\mu)=\{1-m\}\cup [2,+\infty),$ so $1$ jumped $m$ positions to the left. The height of $\mu\backslash \lambda=\mu$ equals $1,$ and $\spin(\mu)=0.$ %jumped over $0$ elements of $M(\lambda).$
Suppose now that $\mu=(1^m).$ Then $M(\mu)=[0,m-1]\cup [m+1,+\infty),$ so $m$ jumped $m$ positions to the left. The height is $m,$ and $\spin(\mu)=m-1.$ %jumped over exactly $m-1$ elements of $M(\lambda).$
More generally, if $\mu=(m-i+1,1^{i-1}),\ 0< i\le m,$ then $M(\mu)=\{i-m\}\cup[1,i-1]\cup [i+1,+\infty),$ so $i$ jumped $m$ positions to the left, the height is $i,$ and $\spin(\mu)=i-1.$ %jumped over exactly $i-1$ elements of $M(\lambda).$
%It is clear that $M(\mu)$ is obtained from $M(\lambda)+i+1=[i+1,+\infty)$ by a jump of $x=m$ to $0$. Since $x$ jumps over $m-1-i$ elements, we have $\hgt(\mu\setminus \lambda)=i+1$, as expected.
\end{example}

\begin{definition}
Given a skew Young diagram $\nu=\mu\backslash\lambda,$ a standard $m$-ribbon tableau is a skew $m$-ribbon diagram of shape $\nu$ together with an order on the ribbons, say $\nu=r_1\sqcup r_2\sqcup\ldots\sqcup r_n,$\ $r_1\prec \ldots\prec r_n,$ such that for any $k$ the partial union $\lambda\sqcup r_1\sqcup\ldots\sqcup r_k$ is a Young diagram (see Figure \ref{Figure:spin and dinv} for an example). Let $\SRT(\nu,m)$ denote the set of standard $m$-ribbon tableaux of shape $\nu.$ The $\spin$ statistics for $m$-ribbon tableau is defined to be equal to the $\spin$ statistic of the underlying $m$--ribbon diagram. 
\end{definition}

\begin{remark}
The definitions of the $\spin$ statistic differ from source to source. Here we follow the notations from \cite{LT2}. In \cite{LLT} the authors used a factor $\frac{1}{2},$ so their $\spin$ is twice less then ours (and possibly half-integer), in \cite{HHLRU} the authors also subtracted the minimal possible value of $\spin$ on the set of $m$-diagrams of a given shape.
\end{remark}

Note that choosing a standard $m$-ribbon tableau of shape $\nu$ is equivalent to choosing a particular way to obtain $M(\mu)$ from $M(\lambda)$ by a sequence of $m$--jumps, i.e. choosing a particular order of jumps.

\begin{definition}
Let $T=\{\nu=r_1\sqcup r_2\sqcup\ldots\sqcup r_n,\prec\}\in \SRT(\nu,m)$ be a standard $m$-ribbon tableau. The {\it content sequence}  $c(T)$ is the sequence of contents of the ribbons of $T$ read in order $\prec.$
\end{definition}

We will also need the notion of a {\it semistandard $m$-ribbon tableau} and the standardization map. Let $\nu=\mu\backslash\lambda$ be a skew Young diagram. Consider an $m$-ribbon diagram $\nu=r_1\sqcup r_2\sqcup\ldots\sqcup r_n,$ and a function $\tau:\{r_1,\ldots,r_n\}\to \BZ_{>0}.$ The function $\tau$ defines a partial order on the set of ribbons. One can refine this order by using the increasing order on contents of the ribbons. More precisely, one says that $r_i\prec r_j$ if either $\tau(r_i)<\tau(r_j),$ or $\tau(r_i)=\tau(r_j)$ and $c(r_i)<c(r_j).$

\begin{definition}
An $m$-ribbon diagram $\nu=r_1\sqcup r_2\sqcup\ldots\sqcup r_n,$ together with a function $\tau:\{r_1,\ldots,r_n\}\to \BZ_{>0}$ is called a {\it semistandard $m$-ribbon tableau} of shape $\nu$ if
\begin{enumerate}
\item The refinement $\prec$ constructed above is a total order on ribbons,
\item $\prec$ defines a standard $m$-ribbon tableau of shape $\nu.$
\end{enumerate}
Let $\SSRT(\nu,m)$ denote the set of semistandard $m$-ribbon tableaux of shape $\nu.$
\end{definition}

The resulting map $\st:\SSRT(\nu,m)\to\SRT(\nu,m)$ is called the {\it standardization map.} Another way to understand semistandard tableaux is to look at the fibers of the map $\st.$ The following lemma is merely a reformulation of the definitions:  

\begin{lemma}
Let $T=\{\nu=r_1\sqcup\ldots\sqcup r_n,\prec\}\in\SYT(\nu,m)$ be a standard $m$-ribbon tableau and $\tau:\{r_1,\ldots,r_n\}\to \BZ_{>0}$ be a function. Then $\{\nu=r_1\sqcup\ldots\sqcup r_n,\tau\}\in\st^{-1}(T)$ iff
\begin{enumerate}
\item $\tau$ is weakly increasing with respect to $\prec,$
\item if $\tau(r_i)=\tau(r_j)$ and $r_i\prec r_j,$ then $c(r_i)<c(r_j).$
\end{enumerate}
\end{lemma}

\begin{remark}
Equivalently, one can say that $\{\nu=r_1\sqcup\ldots\sqcup r_n,\tau\}$ is a semistandard $m$-ribbon tableau if the function $\tau:\{r_1,\ldots,r_n\}\to\BZ_{>0}$ is weakly increasing in rows and columns, and for any $k\in\BZ_{>0}$ the preimage $\tau^{-1}(k)\subset\nu$ is tiled in such a way that the increasing content order defines a standard $m$-ribbon tableau on $\tau^{-1}(k).$ Equivalently, every ribbon of $\tau^{-1}(k)$ starts from the leftmost box of a row of $\tau^{-1}(k).$ Such shapes and tilings are called {\it vertical $m$-ribbon strips} and {\it official tilings} correspondingly (see Figure \ref{Figure: vertical m-strip}).
\end{remark}

\begin{remark}
Note that our definitions differ from those in \cite{HHLRU} by transposition. This is due to the fact that the classical Shuffle conjecture differs from the rational version by the involution $\Omega$ (twist by the sing representation).
\end{remark}

\begin{figure}

\begin{tikzpicture}[scale=0.5]
\draw [thick] (5,0)--(5,1)--(4,1)--(4,5)--(2,5)--(2,6)--(1,6)--(1,7)--(0,7)--(0,2)--(1,2)--(1,1)--(2,1)--(2,0)--(5,0);
\draw [thick] (4,1)--(3,1)--(3,2)--(2,2)--(2,3)--(0,3);
\draw [thick] (4,3)--(3,3)--(3,4)--(1,4)--(1,5)--(0,5);
\draw (2.3,1.3) node {$r_1$};
\draw (2.3,3.3) node {$r_2$};
\draw (1.6,4.6) node {$r_3$};
\end{tikzpicture}

\caption{A vertical $7$-ribbon strip with the official tiling. Note that the contents of the ribbons satisfy $c(r_1)<c(r_2)<c(r_3)$ (the content increases as we move up and/or left). Note also, that ordering the ribbons in the increasing content order $r_1\prec r_2\prec r_3$ makes a valid standard ribbon tableau.}\label{Figure: vertical m-strip}

\end{figure}
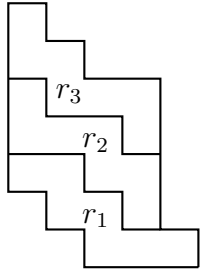
One gets the following corollary for quasisymmetric functions:

\begin{corollary}\label{Corollary:quasisym via semistand}
Let $T\in\SYT(\nu,m),$ and $c(T)$ be the content sequence of $T.$ Then
$$
Q_{\des(c(T))}(z)=\sum\limits_{S\in\st^{-1}(T)} z^S,
$$
where $z^S=z_{\tau(r_1)}z_{\tau(r_2)}\ldots z_{\tau(r_n)}$ for a semistandard tableau $S=\{\nu=r_1\sqcup\ldots\sqcup r_n,\tau\}\in\SSYT(\nu,m).$
\end{corollary}

\subsection{The $m$-cores and quotients}

A Young diagram $\mu$ is called an $m$-core if neither of hook-lengths of its boxes is equal to $m$.
It is clear that $M(\mu)$ is $m$-invariant if and only if $\mu$ is an $m$--core.  More generally, let $M_i(\mu)$ denote the set of elements of $M(\mu)$ 
 with remainder $i$ modulo $m$, define $\widetilde{M_i}(\mu)=(M_i(\mu)-i)/m$. The $m$-quotient of $\mu$ is defined as the $m$-tuple of Young diagrams 
 $\Quot_{m}^{(i)}(\mu)$ corresponding to $\widetilde{M_i}(\mu)$. 
 
Given an arbitrary Young diagram $\mu,$ one can construct the $m$-core of $\mu$ by consecutively removing $m$-strips from it. This is best seen in terms of the subset $M(\mu):$ we move elements to the right by $m$-jumps as much as possible. Clearly, the resulting $m$-invariant subset is independent on the choice of an order of jumps. Let $\Core_m(\mu)$ denote the $m$-core of $\mu.$ By construction, it is clear that the map 
$$
\mu\mapsto \left(\Core_m(\mu),\Quot_m^{(1)}(\mu),\ldots,\Quot_m^{(m)}(\mu)\right)
$$ 
is a bijection between the set of Young diagrams and the set of $m$-cores times the set of $m$-tuples of Young diagrams.

%\begin{example}
%Suppose that for all $i$  the partition $\Quot_{m}^{(i)}(\mu)$ has one part: $\Quot_{m}^{(i)}(\mu)=(s_i)$.  
%\end{example}

Let $D$ be an $m/n$-Dyck path, $\Delta_D\subset \BZ$ be the corresponding $(m,n)$-invariant subset (see Section \ref{Section: Affine permutations}), and $\mu$ be the simultaneous $(m,n)$-core such that $M(\mu)=\Delta_D.$ The map $D\mapsto \mu$ provides a bijection between the set of $m/n$-Dyck paths and the set of simultaneous $(m,n)$-cores. This bijection was first described by J. Anderson in \cite{anderson}, although in somewhat different terms. See also \cite{GM}, Section 2.4.

\section{Main construction}

Let $\om\in\aSmn$ be an $m$-stable affine permutation, 
%$f\in\PF_{m/n}(D)$ be an $m/n$-parking function with the underlying $m/n$-Dyck path $D.$ J. Anderson (\cite{anderson}) constructed a bijection between the set of $m/n$-Dyck paths and the set of simultaneous $(m,n)$--cores. In terms of the constructions described above, this bijection can be described as follows. Let $\om:=\An^{-1}(f)$ be the corresponding stable affine permutation 
define $\Delta:=\{i\in\BZ:\om(i)>0\}$ as before (see Section \ref{Section: Affine permutations}). %Note that by construction, $\Delta_D:=\Delta_{\om}$ only depends on the underlying Dyck path $D.$ 
Since $\Delta$ is $(m,n)$-invariant, it follows that the corresponding Young diagram $\lambda:=M^{-1}(\Delta)$ is a simultaneous $(m,n)$-core.
As before, let $\{a_1,a_2,\ldots,a_n\}$ be the set of $n$-generators of $\Delta:$ 
$$
\Delta\setminus (\Delta+n)=\{\om^{-1}(1),\ldots,\om^{-1}(n)\}=\{a_1,\ldots,a_n\}.
$$
Consider the subset $M=\{a_1-m,\ldots,a_n-m\}\cup (\Delta+n)\subset \BZ$.  It is {\em not} $m$-invariant, hence it corresponds to a Young diagram $\mu$ which is not an $m$-core. By construction, we have
$$
\Core_m(\mu)=\lambda,
$$ 
and the subset $M$ is obtained from $\Delta$ by $n$ jumps $a_1\mapsto a_1-m,\ldots,a_n\mapsto a_n-m.$ One can do the jumps in different orders, however, if $a_i=a_j+m,$ then $a_i$ has to jump before $a_j$. The affine permutation $\om$ prescribes the following order of jumps: we first move $\om^{-1}(1)\mapsto \om^{-1}(1)-m,$ then $\om^{-1}(2)\mapsto \om^{-1}(2)-m,$ and so on up to $\om^{-1}(n)\mapsto \om^{-1}(n)-m.$ The above condition on the order of the jumps is equivalent to $\om$ being $m$-stable. Recall that choosing a valid order of jumps is equivalent to choosing a standard $m$-ribbon tableau of shape $\nu:=\mu\backslash \lambda.$ We conclude that there is a bijection 
\begin{equation}\label{Equation: bijection perm to srt}
T:\aSmn(\Delta)\to\SRT(\nu,m).
\end{equation}
Moreover, the contents of the ribbons of $T(\om)$ are exactly one less then the generators of $\Delta.$ More precisely, one has
$$
c(T(\om))=(\om^{-1}(1)-1,\ldots, \om^{-1}(n)-1),
$$
in particular 
\begin{equation}\label{Equation: omega inv vs cont}
\des(\om^{-1})=\des(c(T(\om))).
\end{equation}

\begin{example}\label{Example: jumps}
Consider the $7$-stable permutation $\om$ from Example \ref{Example: quasi-symm}. We have $\om(-6)=3,\ \om(1)=5,\ \om(3)=1,\ \om(5)=2,$ and $\om(12)=4,$ or
$$
\begin{array}{ccccccccccccccc}
x \ \ &-6&-5&-4&-3&-2&-1&0&1&2&3&4&5&6&7\ \\
\omega(x) \ &3&-8&0&-11&-4&8&-3&5&-6&1&13&2&10&-1\ \\
\end{array}
$$
Therefore $\Delta=\{i\in\BZ:\ \om(i)>0\}=\{-6,-1,1,3,4,5,6\}\sqcup \BZ_{\ge 8}.$ The process of obtaining the subset $M$ from $\Delta$ is best described in terms of the characteristic functions:
$$
\begin{array}{cccccccccccccccccccccccccc}
&&&&&&&{\scriptstyle -6}&&&&&&&{\scriptstyle 1}&&{\scriptstyle 3}&&{\scriptstyle 5}&&&&&&&{\scriptstyle 12}\\
0&0&0&0&0&0&0&{\bf 1_3}&0&0&0&0&1&0&{\bf 1_5}&0&{\bf 1_1}&1&{\bf 1_2}&1&0&1&1&1&1&{\bf 1_4}\\
0&0&0&0&0&0&0&{\bf 1_3}&0&{\bf 1}&0&0&1&0&{\bf 1_5}&0&0&1&{\bf 1_2}&1&0&1&1&1&1&{\bf 1_4}\\
0&0&0&0&0&0&0&{\bf 1_3}&0&{\bf 1}&0&{\bf 1}&1&0&{\bf 1_5}&0&0&1&0&1&0&1&1&1&1&{\bf 1_4}\\
{\bf 1}&0&0&0&0&0&0&0&0&{\bf 1}&0&{\bf 1}&1&0&{\bf 1_5}&0&0&1&0&1&0&1&1&1&1&{\bf 1_4}\\
{\bf 1}&0&0&0&0&0&0&0&0&{\bf 1}&0&{\bf 1}&1&0&{\bf 1_5}&0&0&1&{\bf 1}&1&0&1&1&1&1&0\\
{\bf 1}&0&0&0&0&0&0&{\bf 1}&0&{\bf 1}&0&{\bf 1}&1&0&0&0&0&1&{\bf 1}&1&0&1&1&1&1&0\\
\end{array}
$$
Here the bold $1$'s correspond to the $5$-generators of $\Delta$ and the subscripts are the corresponding values of $\om,$ prescribing the order of jumps. Note that on the first step one ``jumps over'' $2$ elements, on the second -- over $3$ elements, then $0,\ 5,$ and $3$ elements correspondingly. Therefore, the total jump is $13.$ See the corresponding standard $7$-ribbon tableau in Figure \ref{Figure:spin and dinv}.
\end{example}

Let $\om_0\in\aSmn(\Delta)$ be the unique element of $\aSmn(\Delta)$ satisfying 
$$
\om_0^{-1}(1)<\om_0^{-1}(2)<\ldots<\om_0^{-1}(n).
$$ 
It follows that the ribbon tableau $T(\om_0)$ is the tableau corresponding to the increasing order on the jumps. In other words, $\nu$ is always a vertical $m$-ribbon strip, and the underlying tiling of $T(\om)$ is the official tiling of $\nu.$ It turns out that the $\dinv$ and $\spin$ statistics are closely related:

\begin{lemma}\label{Lemma:dinv vs jump}
On has the following formula:
$$
\delta-\dinv(\om)=\frac{1}{2}\left(\spin T(\om)+\spin T(\om_0)\right).
$$
where $\delta:=\frac{(m-1)(n-1)}{2}.$
%where $T_{\min}\in\SRT(\nu,m)$ is the $m$-ribbon standard tableau with the minimal possible spin. 
\end{lemma}

%\begin{remark}
%This Lemma is similar to Lemma 5.2.2 in \cite{HHLRU}, but we would like to present a shorter proof for the reader's convenience. We also compute the constant shift explicitly.
%\end{remark}

\begin{proof}
Let us count how many elements we ``jump over'' as we construct $M$ from $\Delta$ according to the order prescribed by $\om.$  As we move $\om^{-1}(a)$ to $\om^{-1}(a)-m$ it jumps over the elements of the following three types:
\begin{enumerate}
\item $\om^{-1}(b)-m$ for $1\le b<a$ such that $0<\om^{-1}(b)-\om^{-1}(a)<m$ % (so $\om_f(a)-m<\om_f(b)-m<\om_f(a)$)\\
\item $\om^{-1}(b)$ for $a<b\le n$ such that $0<\om^{-1}(a)-\om^{-1}(b)<m$ %(so $\om_f(a)-m<\om_f(b)<\om_f(a)$)\\
\item $k\in\Delta+n,$ such that $0<\om^{-1}(a)-k<m$ %(so $\om_f(a)-m<\om_f(\beta)<\om_f(a)$
\end{enumerate}
Let $N_{1}(\om),N_{2}(\om)$ and $N_3(\om)$ denote the number of pairs $(a,b)$ satisfying (1)--(3), respectively. 
Note that $N_1(\om)=N_2(\om)$ and 
$$
\dinv(\om)=\delta-N_2(\om)-N_3(\om).
$$ 
Note also that $N_1(\om_0)=N_2(\om_0)=0$ and $N_3(\om)=N_3(\om_0)$ for all $\om\in\aSmn(\Delta).$ So, we get
$$
\spin(T(\om))=N_{1}(\om)+N_{2}(\om)+N_3(\om)=2N_{2}(\om)+N_3(\om),
$$
and
$$
\spin(T(\om_0))=N_3(\om_0)=N_3(\om).
$$
We conclude that 
$$
\delta-\dinv(\om)=N_2(\om)+N_3(\om)=\frac{1}{2}\left(\spin T(\om)+\spin T(\om_0)\right).
$$

\end{proof}

\begin{corollary}\label{Corollary: spin via dinv}
One has the following equations:
$$
\spin T(\om_0)=\delta-\dinv(\om_0),\ \spin T(\om)= \delta+\dinv(\om_0)-2\dinv(\om),
$$
$$
\dinv(\om)=\frac{1}{2}(\delta+\dinv(\om_0)-\spin T(\om)).
$$
\end{corollary}

\begin{figure}
\begin{tikzpicture}[scale=0.5]
\filldraw[fill=gray,thick] (0,0)--(7,0)--(7,1)--(3,1)--(3,2)--(2,2)--(2,3)--(1,3)--(1,7)--(0,7)--(0,0);
%\draw [dashed] (0,1)--(7,1);
%\draw [dashed] (0,2)--(7,2);
%\draw [dashed] (0,3)--(7,3);
%\draw [dashed] (0,4)--(6,4);
%\draw [dashed] (0,5)--(6,5);
%\draw [dashed] (0,6)--(2,6);
%\draw [dashed] (0,7)--(2,7);
%\draw [dashed] (0,8)--(2,8);
%\draw [dashed] (0,9)--(1,9);
%\draw [dashed] (0,10)--(1,10);
%\draw [dashed] (0,11)--(1,11);
%\draw [dashed] (0,12)--(1,12);
%\draw [dashed] (1,0)--(1,8);
%\draw [dashed] (2,0)--(2,8);
%\draw [dashed] (3,0)--(3,5);
%\draw [dashed] (4,0)--(4,5);
%\draw [dashed] (5,0)--(5,5);
%\draw [dashed] (6,0)--(6,5);
%\draw [dashed] (7,0)--(7,3);
%\draw [dashed] (8,0)--(8,2);
%\draw [dashed] (9,0)--(9,1);
%\draw [dashed] (10,0)--(10,1);
%\draw [dashed] (11,0)--(11,1);
%\draw [dashed] (12,0)--(12,1);
%\draw [dashed] (13,0)--(13,1);
\draw [thick] (7,0)--(7,1)--(14,1)--(14,0)--(7,0);
\draw [thick] (0,7)--(0,12)--(1,12)--(1,8)--(2,8)--(2,6)--(1,6)--(1,7);
\draw [thick] (1,6)--(2,6)--(2,5)--(4,5)--(4,4)--(5,4)--(5,2)--(4,2)--(4,3)--(3,3)--(3,4)--(1,4);
\draw [thick] (5,2)--(6,2)--(6,1);
\draw [thick] (4,5)--(6,5)--(6,3)--(7,3)--(7,2)--(8,2)--(8,1);
\draw (2.5,2.5) node {$1$};
\draw (3.5,4.5) node {$2$};
\draw (10.5,0.5) node {$3$};
\draw (0.5,9.5) node {$4$};
\draw (5.5,2.5) node {$5$};
\draw (5,-0.5) node {$\spin(T(\om))=13$};
\end{tikzpicture}\quad
\begin{tikzpicture}[scale=0.5]
\filldraw[fill=gray,thick] (0,0)--(7,0)--(7,1)--(3,1)--(3,2)--(2,2)--(2,3)--(1,3)--(1,7)--(0,7)--(0,0);
%\draw [dashed] (0,1)--(7,1);
%\draw [dashed] (0,2)--(7,2);
%\draw [dashed] (0,3)--(7,3);
%\draw [dashed] (0,4)--(6,4);
%\draw [dashed] (0,5)--(6,5);
%\draw [dashed] (0,6)--(2,6);
%\draw [dashed] (0,7)--(2,7);
%\draw [dashed] (0,8)--(2,8);
%\draw [dashed] (0,9)--(1,9);
%\draw [dashed] (0,10)--(1,10);
%\draw [dashed] (0,11)--(1,11);
%\draw [dashed] (0,12)--(1,12);
%\draw [dashed] (1,0)--(1,8);
%\draw [dashed] (2,0)--(2,8);
%\draw [dashed] (3,0)--(3,5);
%\draw [dashed] (4,0)--(4,5);
%\draw [dashed] (5,0)--(5,5);
%\draw [dashed] (6,0)--(6,5);
%\draw [dashed] (7,0)--(7,3);
%\draw [dashed] (8,0)--(8,2);
%\draw [dashed] (9,0)--(9,1);
%\draw [dashed] (10,0)--(10,1);
%\draw [dashed] (11,0)--(11,1);
%\draw [dashed] (12,0)--(12,1);
%\draw [dashed] (13,0)--(13,1);

\draw [thick] (7,0)--(7,1)--(14,1)--(14,0)--(7,0);
\draw [thick] (0,7)--(0,12)--(1,12)--(1,8)--(2,8)--(2,6)--(1,6)--(1,7);
\draw [thick]  (2,6)--(2,5)--(6,5)--(6,3);
\draw [thick] (7,2)--(7,3)--(5,3)--(5,4)--(1,4);
\draw [thick]  (2,3)--(4,3)--(4,2)--(8,2)--(8,1);
\draw (10.5,0.5) node {$1$};
\draw (4.5,1.5) node {$2$};
\draw (4.5,3.5) node {$3$};
\draw (3.5,4.5) node {$4$};
\draw (0.5,9.5) node {$5$};
\draw (5,-0.5) node {$\spin(T(\om_0))=9$};
\end{tikzpicture}
\caption{Standard ribbon tableaux for $7$--stable permutations $\om$ (left) and $\om_0$ (right).}
\label{Figure:spin and dinv}
\end{figure}
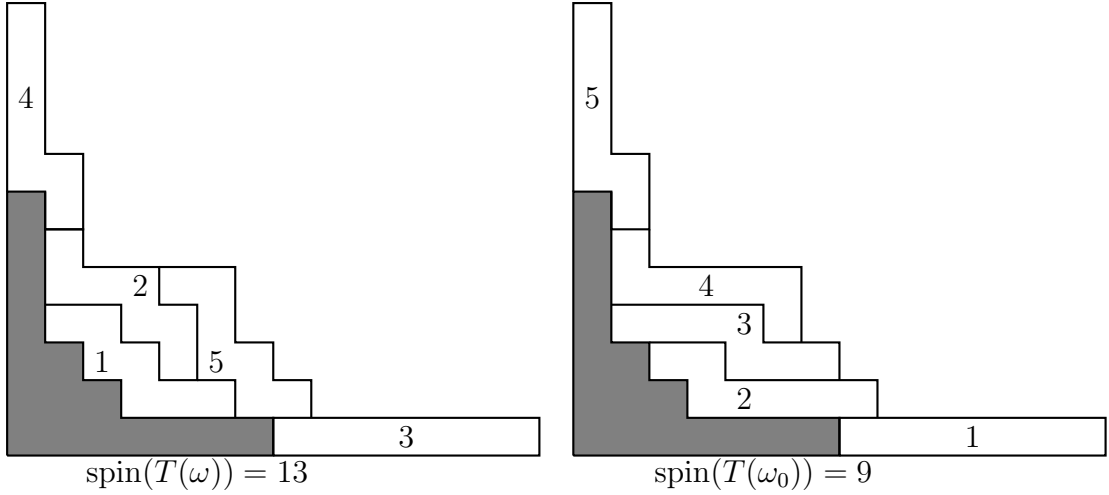

\begin{example}
We illustrate the standard $7$-ribbon tableau corresponding to Example \ref{Example: jumps} in Figure \ref{Figure:spin and dinv} on the left. On the right we have the $7$-ribbon tableau of the same shape with the ribbons ordered with respect to contents. It corresponds to the increasing order of jumps. Observe that $\delta=12,\ \dinv(\om_0)=12-\spin T(\om_0)=3,$ and 
$$
12-\dinv(\om)=\frac{1}{2}(\spin T(\om_0)+\spin T(\om))=11,
$$ 
so $\dinv(\om)=1,$ which matches the answer in Example \ref{Example: quasi-symm}.
\end{example}

%\begin{proof}
%By Lemma \ref{Lemma:jump vs spin} one has
%$$
%\jump(T(\om))=\frac{m-1}{m}|T(\om)|-2\spin(T(\om))=(m-1)n-2\spin(T(\om)).
%$$
%Also, plugging $\om=\om_0$ in Lemma \ref{Lemma:dinv vs jump} one gets
%$$
%\jump(T(\om_0))=\frac{(m-1)(n-1)}{2}-\dinv(\om_0).
%$$
%Finally, plugging the above into Lemma \ref{Lemma:dinv vs jump} yields
%$$
%\frac{(m-1)(n-1)}{2}-\dinv(\om)=\frac{1}{2}\left((m-1)n-2\spin(T(\om))+\frac{(m-1)(n-1)}{2}-\dinv(\om_0)\right)
%$$
%$$
%=\frac{1}{2}\left( \frac{(m-1)(3n-1)}{2}-\dinv(\om_0)\right)-\spin(T(\om)),
%$$
%which is equivalent to 
%$$
%\spin(T(\om))=\dinv(\om)+\frac{1}{2}\left( \frac{(m-1)(3n-1)}{2}-\dinv(\om_0)\right)-\frac{(m-1)(n-1)}{2}
%$$
%$$
%=\dinv(\om)+\frac{1}{2}\left(\frac{(m-1)(n+1)}%{2}-\dinv(\om_0)\right).
%$$
%\end{proof}

\begin{proof}[Proof of Theorem \ref{thm:main}]

We follow the logic of \cite{HHLRU}. Combining equations \eqref{Equation: Fdt via permutations}, \eqref{Equation: bijection perm to srt}, \eqref{Equation: omega inv vs cont} and Corollary \ref{Corollary: spin via dinv}, we get:

%By Lemma \ref{lem:core}, $\Core_m(\mu')=\mu$ and a semistandard tableau of shape $(D+(1^n))\setminus D$ corresponds to a semistandard ribbon tableau of shape $\mu'\setminus\mu$.  Similarly to  [(eq.35)]\cite{HHLRU} we can write
$$
\CF(D;t)=\sum_{f\in \PF_{m/n}(D)}t^{\dinv(f)}\cdot Q_{\ides(f)}(z)=\sum_{\om\in\aSmn(\Delta_D)}t^{\dinv(\om)}Q_{\des(\om^{-1})}(z)
$$
$$
=t^{e(D)}\sum_{T\in\SRT(\nu,m)} t^{-\frac{1}{2}\spin(T)}Q_{\des(c(T))}(z),
$$
where $e(D)=\frac{1}{2}(\delta+\dinv(\om_0)).$ Applying Corollary \ref{Corollary:quasisym via semistand} we get
$$
t^{e(D)}\sum_{T\in\SRT(\nu,m)} t^{-\frac{1}{2}\spin(T)}Q_{\des(c(T))}(z)=t^{e(D)}\sum_{T\in\SSRT(\nu,m)} t^{-\frac{1}{2}\spin(T)}z^T.
$$
Finally, by \cite[Proposition 5.3.1]{HHLRU} one has
$$
(t^{-e(D)}\CF(D,t),s_{\kappa}(z))=P^{-}_{\mu+\rho,\lambda+m\kappa+\rho}(t),
$$
where $P^{-}_{\mu+\rho,\lambda+m\kappa+\rho}(t)$ denotes the parabolic affine Kazhdan-Lusztig polynomial
in the sense of \cite{LT1,LT2}. This polynomial is known to have nonnegative coefficients by the work of Kashiwara-Tanisaki and  Shan \cite{KT,Shan}.
\end{proof}

\section{Example: $m=2,n=5$}
\label{Example: ribbon tableaux 2,5}

\begin{figure}[!ht]
\begin{tikzpicture}[scale=0.7]
\filldraw[fill=gray,thick] (0,0)--(2,0)--(2,1)--(1,1)--(1,2)--(0,2)--(0,0);
\draw[thick] (2,0)--(4,0)--(4,1)--(2,1);
\draw[thick] (3,1)--(3,2)--(1,2);
\draw[thick] (2,2)--(2,3)--(0,3)--(0,2);
\draw[thick] (1,3)--(1,5)--(0,5)--(0,3);
\draw[thick] (1,5)--(1,7)--(0,7)--(0,5);
\draw (3,0.5) node {$1_{-2}$};
\draw (2,1.5) node {$2_0$};
\draw (1,2.5) node {$3_2$};
\draw (0.5,4) node {$4_4$};
\draw (0.5,6) node {$5_6$};
%\draw (3,-0.5) node {$\spin(T)=2,\dinv(\om)=0$};
\draw (2,-0.7) node {$c(T)=(-2,0,2,4,6),\ \des(T)=\emptyset$};
\end{tikzpicture}\quad
\begin{tikzpicture}[scale=0.7]
\filldraw[fill=gray,thick] (0,0)--(1,0)--(1,1)--(0,1)--(0,0);
\draw[thick] (1,0)--(3,0)--(3,1)--(1,1);
\draw[thick] (2,1)--(2,2)--(0,2)--(0,1);
\draw[thick] (2,2)--(2,3)--(0,3)--(0,2);
\draw[thick] (2,3)--(2,4)--(0,4)--(0,3);
\draw[thick] (1,4)--(1,6)--(0,6)--(0,4);
\draw (2,0.5) node {$1_{-1}$};
\draw (1,1.5) node {$2_1$};
\draw (1,2.5) node {$3_2$};
\draw (1,3.5) node {$4_3$};
\draw (0.5,5) node {$5_5$};
%\draw (3,-0.5) node {$\spin(T)=1,\dinv(\om)=1$};
\draw (2,-0.7) node {$c(T)=(-1,1,2,3,5),\ \des(T)=\emptyset$};
\end{tikzpicture}
\caption{On the left: tableau with $\area(\om)=2,\ \spin(T)=\spin(T(\om_0))=2,$ and $\dinv(\om)=0.$ On the right: tableau with $\area(\om)=1,\ \spin(T)=\spin(T(\om_0))=1,$ and $\dinv(\om)=1.$ Subscripts indicate the contents of the ribbons.}\label{Figure: tableaux 20 11}
\end{figure}
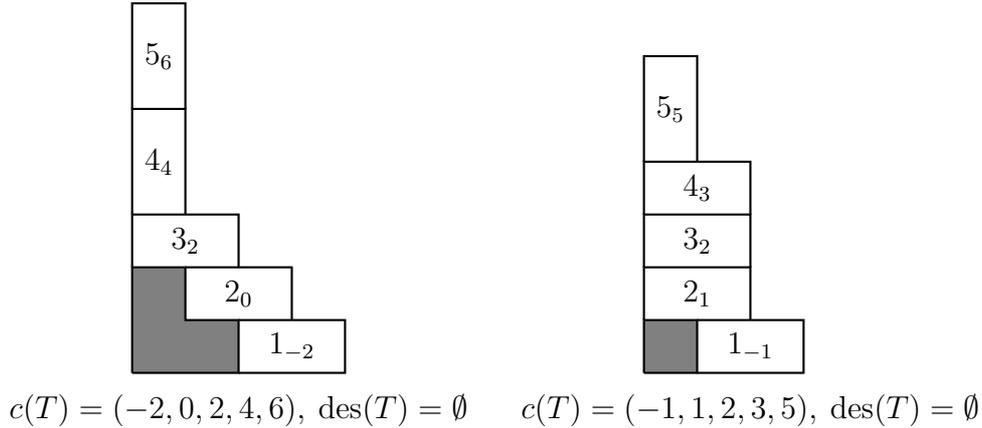

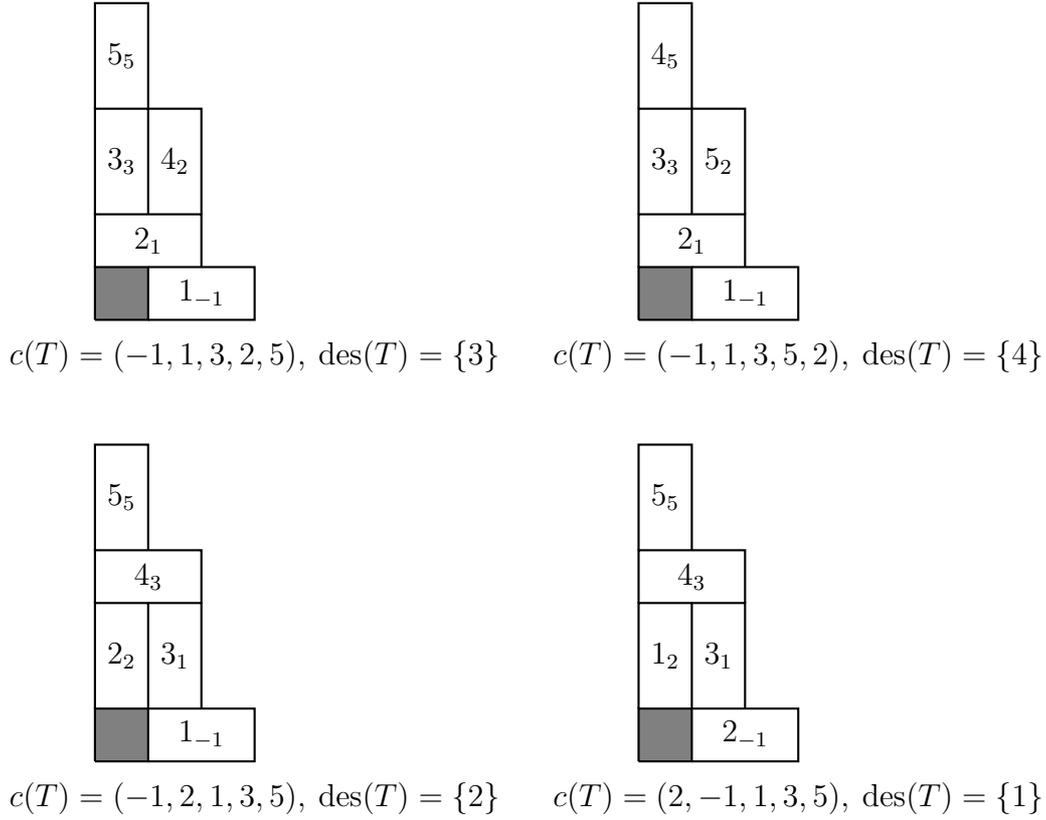
\begin{figure}[!ht]
\begin{tikzpicture}[scale=0.7]
\filldraw[fill=gray,thick] (0,0)--(1,0)--(1,1)--(0,1)--(0,0);
\draw[thick] (1,0)--(3,0)--(3,1)--(1,1);
\draw[thick] (2,1)--(2,2)--(0,2)--(0,1);
\draw[thick] (1,2)--(1,4)--(0,4)--(0,2);
\draw[thick] (2,2)--(2,4)--(1,4);
\draw[thick] (1,4)--(1,6)--(0,6)--(0,4);
\draw (2,0.5) node {$1_{-1}$};
\draw (1,1.5) node {$2_1$};
\draw (0.5,3) node {$3_3$};
\draw (1.5,3) node {$4_2$};
\draw (0.5,5) node {$5_5$};
%\draw (3,-0.5) node {$\spin(T)=3,\dinv(\om)=0$};
\draw (3,-0.7) node {$c(T)=(-1,1,3,2,5),\ \des(T)=\{3\}$};
\end{tikzpicture}\quad
\begin{tikzpicture}[scale=0.7]
\filldraw[fill=gray,thick] (0,0)--(1,0)--(1,1)--(0,1)--(0,0);
\draw[thick] (1,0)--(3,0)--(3,1)--(1,1);
\draw[thick] (2,1)--(2,2)--(0,2)--(0,1);
\draw[thick] (1,2)--(1,4)--(0,4)--(0,2);
\draw[thick] (2,2)--(2,4)--(1,4);
\draw[thick] (1,4)--(1,6)--(0,6)--(0,4);
\draw (2,0.5) node {$1_{-1}$};
\draw (1,1.5) node {$2_1$};
\draw (0.5,3) node {$3_3$};
\draw (1.5,3) node {$5_2$};
\draw (0.5,5) node {$4_5$};
%\draw (3,-0.5) node {$\spin(T)=3,\dinv(\om)=0$};
\draw (3,-0.7) node {$c(T)=(-1,1,3,5,2),\ \des(T)=\{4\}$};
\end{tikzpicture}

\vskip 0.3in

\begin{tikzpicture}[scale=0.7]
\filldraw[fill=gray,thick] (0,0)--(1,0)--(1,1)--(0,1)--(0,0);
\draw[thick] (1,0)--(3,0)--(3,1)--(1,1);
\draw[thick] (1,1)--(1,3)--(0,3)--(0,1);
\draw[thick] (2,1)--(2,3)--(1,3);
\draw[thick] (2,3)--(2,4)--(0,4)--(0,3);
\draw[thick] (1,4)--(1,6)--(0,6)--(0,4);
\draw (2,0.5) node {$1_{-1}$};
\draw (0.5,2) node {$2_2$};
\draw (1.5,2) node {$3_1$};
\draw (1,3.5) node {$4_3$};
\draw (0.5,5) node {$5_5$};
%\draw (3,-0.5) node {$\spin(T)=3,\dinv(\om)=0$};
\draw (3,-0.7) node {$c(T)=(-1,2,1,3,5),\ \des(T)=\{2\}$};
\end{tikzpicture}\quad
\begin{tikzpicture}[scale=0.7]
\filldraw[fill=gray,thick] (0,0)--(1,0)--(1,1)--(0,1)--(0,0);
\draw[thick] (1,0)--(3,0)--(3,1)--(1,1);
\draw[thick] (1,1)--(1,3)--(0,3)--(0,1);
\draw[thick] (2,1)--(2,3)--(1,3);
\draw[thick] (2,3)--(2,4)--(0,4)--(0,3);
\draw[thick] (1,4)--(1,6)--(0,6)--(0,4);
\draw (2,0.5) node {$2_{-1}$};
\draw (0.5,2) node {$1_2$};
\draw (1.5,2) node {$3_1$};
\draw (1,3.5) node {$4_3$};
\draw (0.5,5) node {$5_5$};
%\draw (3,-0.5) node {$\spin(T)=3,\dinv(\om)=0$};
\draw (3,-0.7) node {$c(T)=(2,-1,1,3,5),\ \des(T)=\{1\}$};
\end{tikzpicture}\quad
\caption{Tableaux with $\area(\om)=1,\ \spin(T)=3,\ \spin(T(\om_0))=1,$ and $\dinv(\om)=0.$ Subscripts indicate the contents of the ribbons.}
\label{Figure: tableaux 10}
\end{figure}

In \cite{GMV}  all $2$--stable affine permutations in $\widetilde{S}^2_5$ were listed together with their $\dinv$ and $\area$ statistics.
In this section, we show the corresponding 2--ribbon (domino) tableaux. The $\spin$ statistic of a domino tableau just equals the number of vertical dominoes, so it is particularly easy to visualize. We have $\delta=2.$ Figures \ref{Figure: tableaux 20 11}, \ref{Figure: tableaux 10}, \ref{Figure:tableaux 01}, and \ref{Figure:tableaux 00} show all $16$ possible ribbon tableaux. The subscripts indicate the contents of the ribbons.

There is a unique $2/5$--parking function with area $2$, the corresponding Dyck path $D_2$ is empty. The corresponding domino tableau is shown on the left in Figure \ref{Figure: tableaux 20 11} and has $\dinv=0$ and $\des=\emptyset.$ Therefore,
$$
\CF(D_2;t)=Q_\emptyset(z)=\sum\limits_{i_1\le i_2\le i_3\le i_4\le i_5} z_{i_1}z_{i_2}z_{i_3}z_{i_4}z_{i_5}=h_5=s_5.
$$
%where $z_I=z_{i_1}z_{i_2}z_{i_3}z_{i_4}z_{i_5}.$
There are five $2/5$--parking functions with area $1.$ One of them has $\dinv=1$ and $\des=\emptyset$ (see Figure \ref{Figure: tableaux 20 11}), and the rest have $\dinv=0$ and the descent sets $\{1\},\{2\},\{3\},$ and $\{4\}$ correspondingly (see Figure \ref{Figure: tableaux 10}). Therefore,
we have:
$$
\CF(D_1;t)=tQ_\emptyset(z)+Q_{\{1\}}(z)+Q_{\{2\}}(z)+Q_{\{3\}}(z)+Q_{\{4\}}(z)
$$
%$$
%=ts_5+\sum\limits_{i_1< i_2\le i_3\le i_4\le i_5} z_I+\sum\limits_{i_1\le i_2< i_3\le i_4\le i_5} z_I+\sum\limits_{i_1\le i_2\le i_3< i_4\le i_5} z_I+\sum\limits_{i_1\le i_2\le i_3\le i_4< i_5} z_I
%$$
$$
=ts_5+s_{4,1}
$$
Finally, there are ten $2/5$--parking functions with area $0:$ one of them has $\dinv=2$ and $\des=\emptyset,$ four has $\dinv=1$ and descent sets $\{1\},\{2\},\{3\},$ and $\{4\}$ correspondingly (see Figure \ref{Figure:tableaux 01}), and the remaining five have $\dinv=0$ and descent sets $\{1,3\},\{1,4\},\{2,4\},\{2\}$ and $\{3\}$ correspondingly (see Figure \ref{Figure:tableaux 00}). Therefore,
$$
\CF(D_0;t)=t^2Q_\emptyset+t(Q_{\{1\}}+Q_{\{2\}}+Q_{\{3\}}+Q_{\{4\}})+Q_{\{1,3\}}+Q_{\{1,4\}}+Q_{\{2,4\}}+Q_{\{2\}}+Q_{\{3\}}
$$
$$
=t^2s_5+ts_{1,4}+s_{3,2}.
$$
%\begin{remark}
%For the identities between Gessel's quasisymmetric functions and Schur functions see \cite{G}.
%\end{remark}

Finally, we get
$$
\CF_{2/5}(q,t)=(q^2+qt+t^2)s_{5}+(q+t)s_{4,1}+s_{3,2}.
$$

%Note, that to get the characteristic of the parking function module one should simply plug $t=1:$

%$$
%\CF_{\PF_{2/5}}(q)=(q^2+q+1)s_{5}+(q+1)s_{4,1}+s_{3,2}.
%$$

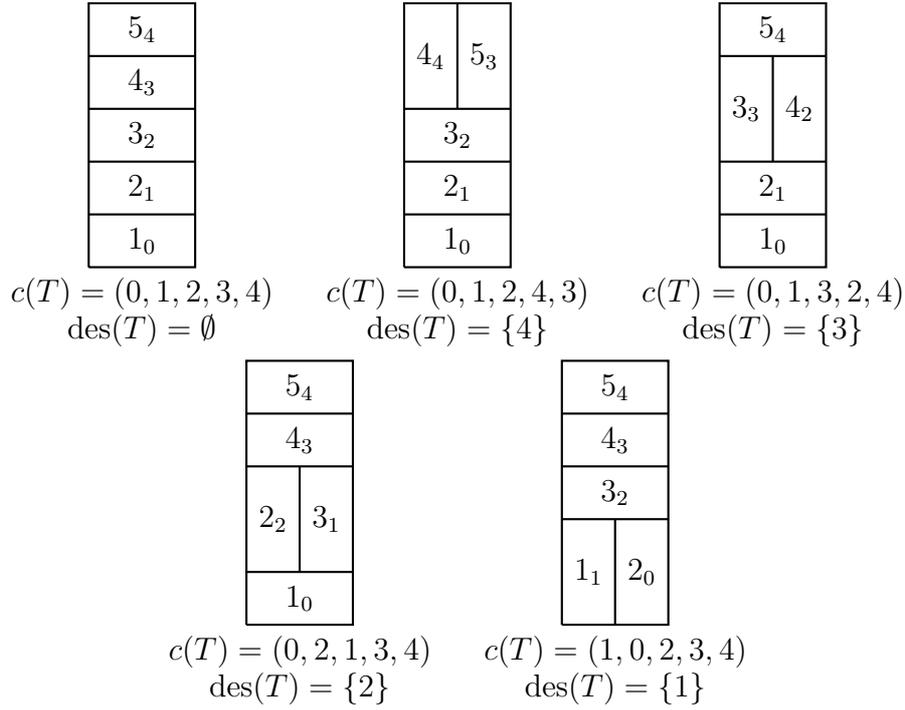
\begin{figure}[!ht]
\begin{tikzpicture}[scale=0.7]
\draw[thick] (0,0)--(2,0)--(2,5)--(0,5)--(0,0);
\draw[thick] (2,1)--(0,1);
\draw[thick] (2,2)--(0,2);
\draw[thick] (2,3)--(0,3);
\draw[thick] (2,4)--(0,4);
\draw (1,0.5) node {$1_0$};
\draw (1,1.5) node {$2_1$};
\draw (1,2.5) node {$3_2$};
\draw (1,3.5) node {$4_3$};
\draw (1,4.5) node {$5_4$};
\draw (1,-0.5) node {$c(T)=(0,1,2,3,4)$};
\draw (1,-1.2) node {$\des(T)=\emptyset$};
\end{tikzpicture}\quad
\begin{tikzpicture}[scale=0.7]
\draw[thick] (0,0)--(2,0)--(2,5)--(0,5)--(0,0);
\draw[thick] (2,1)--(0,1);
\draw[thick] (2,2)--(0,2);
\draw[thick] (2,3)--(0,3);
\draw[thick] (1,3)--(1,5);
\draw (1,0.5) node {$1_0$};
\draw (1,1.5) node {$2_1$};
\draw (1,2.5) node {$3_2$};
\draw (0.5,4) node {$4_4$};
\draw (1.5,4) node {$5_3$};
\draw (1,-0.5) node {$c(T)=(0,1,2,4,3)$};
\draw (1,-1.2) node {$\des(T)=\{4\}$};
\end{tikzpicture}\quad
\begin{tikzpicture}[scale=0.7]
\draw[thick] (0,0)--(2,0)--(2,5)--(0,5)--(0,0);
\draw[thick] (2,1)--(0,1);
\draw[thick] (2,2)--(0,2);
\draw[thick] (1,2)--(1,4);
\draw[thick] (2,4)--(0,4);
\draw (1,0.5) node {$1_0$};
\draw (1,1.5) node {$2_1$};
\draw (0.5,3) node {$3_3$};
\draw (1.5,3) node {$4_2$};
\draw (1,4.5) node {$5_4$};
\draw (1,-0.5) node {$c(T)=(0,1,3,2,4)$};
\draw (1,-1.2) node {$\des(T)=\{3\}$};
\end{tikzpicture}\quad
\begin{tikzpicture}[scale=0.7]
\draw[thick] (0,0)--(2,0)--(2,5)--(0,5)--(0,0);
\draw[thick] (2,1)--(0,1);
\draw[thick] (1,1)--(1,3);
\draw[thick] (2,3)--(0,3);
\draw[thick] (2,4)--(0,4);
\draw (1,0.5) node {$1_0$};
\draw (0.5,2) node {$2_2$};
\draw (1.5,2) node {$3_1$};
\draw (1,3.5) node {$4_3$};
\draw (1,4.5) node {$5_4$};
\draw (1,-0.5) node {$c(T)=(0,2,1,3,4)$};
\draw (1,-1.2) node {$\des(T)=\{2\}$};
\end{tikzpicture}\quad
\begin{tikzpicture}[scale=0.7]
\draw[thick] (0,0)--(2,0)--(2,5)--(0,5)--(0,0);
\draw[thick] (1,0)--(1,2);
\draw[thick] (2,2)--(0,2);
\draw[thick] (2,3)--(0,3);
\draw[thick] (2,4)--(0,4);
\draw (0.5,1) node {$1_1$};
\draw (1.5,1) node {$2_0$};
\draw (1,2.5) node {$3_2$};
\draw (1,3.5) node {$4_3$};
\draw (1,4.5) node {$5_4$};
\draw (1,-0.5) node {$c(T)=(1,0,2,3,4)$};
\draw (1,-1.2) node {$\des(T)=\{1\}$};
\end{tikzpicture}\quad
\caption{On the left: tableau with $\area(\om)=0,\ \spin(T)=\spin(T(\om_0))=0,$ and $\dinv(\om)=2.$ The rest: $4$ tableaux with $\area(\om)=0,\ \spin(T)=2,\ \spin(T(\om))=0,$ and $\dinv(\om)=1.$ Subscripts indicate the contents of the ribbons.}\label{Figure:tableaux 01}
\end{figure}

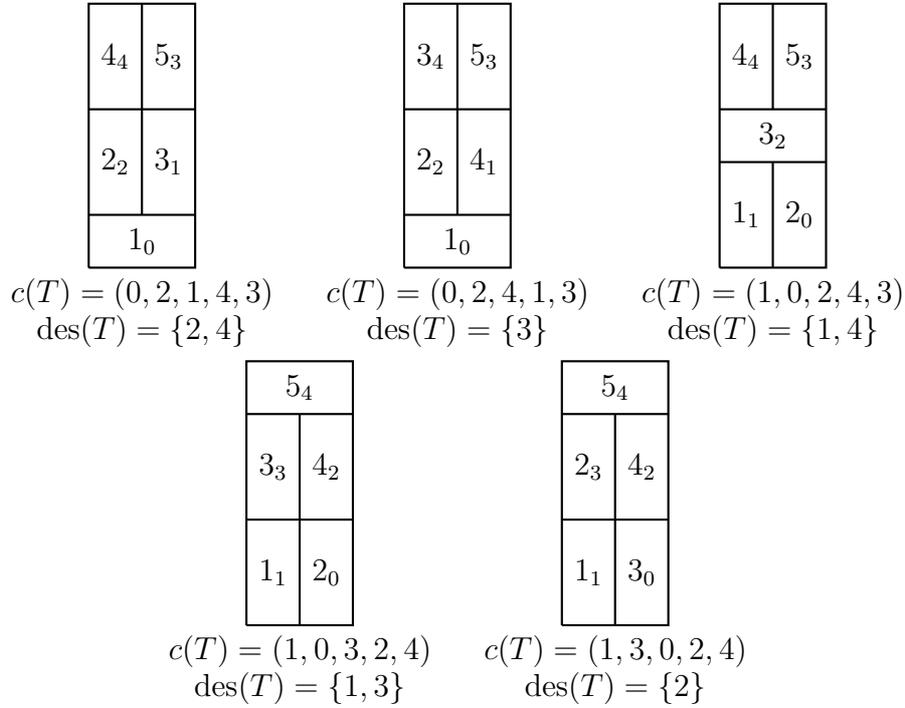
\begin{figure}[!ht]
\begin{tikzpicture}[scale=0.7]
\draw[thick] (0,0)--(2,0)--(2,5)--(0,5)--(0,0);
\draw[thick] (2,1)--(0,1);
\draw[thick] (1,1)--(1,3);
\draw[thick] (2,3)--(0,3);
\draw[thick] (1,3)--(1,5);
\draw (1,0.5) node {$1_0$};
\draw (0.5,2) node {$2_2$};
\draw (1.5,2) node {$3_1$};
\draw (0.5,4) node {$4_4$};
\draw (1.5,4) node {$5_3$};
\draw (1,-0.5) node {$c(T)=(0,2,1,4,3)$};
\draw (1,-1.2) node {$\des(T)=\{2,4\}$};
\end{tikzpicture}\quad
\begin{tikzpicture}[scale=0.7]
\draw[thick] (0,0)--(2,0)--(2,5)--(0,5)--(0,0);
\draw[thick] (2,1)--(0,1);
\draw[thick] (1,1)--(1,3);
\draw[thick] (2,3)--(0,3);
\draw[thick] (1,3)--(1,5);
\draw (1,0.5) node {$1_0$};
\draw (0.5,2) node {$2_2$};
\draw (1.5,2) node {$4_1$};
\draw (0.5,4) node {$3_4$};
\draw (1.5,4) node {$5_3$};
\draw (1,-0.5) node {$c(T)=(0,2,4,1,3)$};
\draw (1,-1.2) node {$\des(T)=\{3\}$};
\end{tikzpicture}\quad
\begin{tikzpicture}[scale=0.7]
\draw[thick] (0,0)--(2,0)--(2,5)--(0,5)--(0,0);
\draw[thick] (1,0)--(1,2);
\draw[thick] (2,2)--(0,2);
\draw[thick] (2,3)--(0,3);
\draw[thick] (1,3)--(1,5);
\draw (0.5,1) node {$1_1$};
\draw (1.5,1) node {$2_0$};
\draw (1,2.5) node {$3_2$};
\draw (0.5,4) node {$4_4$};
\draw (1.5,4) node {$5_3$};
\draw (1,-0.5) node {$c(T)=(1,0,2,4,3)$};
\draw (1,-1.2) node {$\des(T)=\{1,4\}$};
\end{tikzpicture}\quad
\begin{tikzpicture}[scale=0.7]
\draw[thick] (0,0)--(2,0)--(2,5)--(0,5)--(0,0);
\draw[thick] (1,0)--(1,2);
\draw[thick] (2,2)--(0,2);
\draw[thick] (1,2)--(1,4);
\draw[thick] (2,4)--(0,4);
\draw (0.5,1) node {$1_1$};
\draw (1.5,1) node {$2_0$};
\draw (0.5,3) node {$3_3$};
\draw (1.5,3) node {$4_2$};
\draw (1,4.5) node {$5_4$};
\draw (1,-0.5) node {$c(T)=(1,0,3,2,4)$};
\draw (1,-1.2) node {$\des(T)=\{1,3\}$};
\end{tikzpicture}\quad
\begin{tikzpicture}[scale=0.7]
\draw[thick] (0,0)--(2,0)--(2,5)--(0,5)--(0,0);
\draw[thick] (1,0)--(1,2);
\draw[thick] (2,2)--(0,2);
\draw[thick] (1,2)--(1,4);
\draw[thick] (2,4)--(0,4);
\draw (0.5,1) node {$1_1$};
\draw (1.5,1) node {$3_0$};
\draw (0.5,3) node {$2_3$};
\draw (1.5,3) node {$4_2$};
\draw (1,4.5) node {$5_4$};
\draw (1,-0.5) node {$c(T)=(1,3,0,2,4)$};
\draw (1,-1.2) node {$\des(T)=\{2\}$};
\end{tikzpicture}
\caption{Tableaux with $\area(\om)=0, \spin(T)=4, \ \spin(T(\om_0))=0,$ and $\dinv(\om)=0.$ Subscripts indicate the contents of the ribbons.}\label{Figure:tableaux 00}
\end{figure}

\end{document}